\documentclass[12pt,reqno]{amsart}
\usepackage{url}
\let\oldlabel=\label
\def\prellabel{\marginparsep=1em\marginparwidth=44pt
  \def\label##1{\oldlabel{##1}\ifmmode\else\ifinner\else
         \marginpar{{\footnotesize\ \\ \tt
                    ##1}}\fi\fi}}
\def\NZQ{\mathbb}               
\def\NN{{\NZQ N }}
\def\QQ{{\NZQ Q}}
\def\ZZ{{\NZQ Z}}

\def\opn#1#2{\def#1{\operatorname{#2}}} 
\opn\chara{char}
\opn\rank{rank}
\opn\hilb{Hilb}
\opn\gr{gr}
\opn\Rees{{\mathcal R}}
\usepackage[latin1]{inputenc}
\usepackage{mathptmx}
\newtheorem{theorem}{Theorem}[section]
\newtheorem{lemma}[theorem]{Lemma}
\newtheorem{corollary}[theorem]{Corollary}

\theoremstyle{definition}
\newtheorem{remark}[theorem]{Remark}

\newtheorem{example}[theorem]{Example}
\newtheorem{remark/example}[theorem]{Remark/Example}

\newtheorem*{theorem*}{Theorem}

\let\epsilon=\varepsilon
\let\phi=\varphi
\let\kappa=\varkappa

\opn\ini{in}
\opn\KRS{KRS}
\opn\krs{krs}
\opn\Krs{Krs}
\opn\DEL{DEL}
\opn\diag{diag}
\opn\Ker{Ker}
\opn\Image{Im}
\opn\DD{{\mathcal D}}
\opn\SS{{\mathcal S}}
\opn\MM{{\mathcal M}}
\opn\GL{GL}
\opn{\hht}{ht}
\opn\Cl{Cl}
\opn\cl{cl}
\opn\height{height}
\opn\reg{reg}
\opn\Reg{Reg}
\opn\pd{pd}
\opn\supp{supp}
\opn{\HS}{HS}

\def\cR{{\mathcal R}}

\def\mm{{\mathfrak m}}

\def\verteq{|\hspace{-2pt}|}

\def\addots{\mathinner{\mkern1mu\raise1pt\hbox{.}\mkern2mu\raise4pt\hbox{.}
        \mkern2mu\raise7pt\vbox{\kern7pt\hbox{.}}\mkern1mu}}

\numberwithin{equation}{section}

\author{Winfried Bruns}
\address{Universit\"at Osnabr\"uck, Institut f\"ur Mathematik, 49069 Osnabr\"uck, Germany}
\email{wbruns@uos.de}

\author{Aldo Conca}
\address{ Dipartimento di Matematica,
Universit\`a degli Studi di Genova, Italy}
\email{conca@dima.unige.it}

\title{A remark on regularity of powers and products of ideals}

\dedicatory{To the memory of our friend Tony Geramita}
 

\date{}

\begin{document}

\begin{abstract}
We give a simple proof for the fact that the Castelnuovo-Mumford regularity and related invariants of products of powers of ideals are asymptotically linear in the exponents, provided that each ideal is generated by elements of constant degree. We provide examples showing that the asymptotic linearity is false in general. On the other hand, the regularity is always given by the maximum of finitely many linear functions whose coefficients belong to the set of the degrees of generators of the ideals. 
\end{abstract}

\maketitle


\section{Introduction}
Let $I$ be an homogeneous ideal of a polynomial ring $R=K[x_1,\dots, x_n]$. Cutkosky, Herzog and Trung \cite{CHT}  and, independently, Kodiyalam \cite{Kod} proved in   that the Castelnuovo-Mumford regularity $\reg(I^a)$ of the powers of $I$ is a linear function of $a\in \NN$ for large $a$. In \cite[Remark pg.252]{CHT}  it is  asserted  that the same result holds for products of powers of ideals $I_1,\dots, I_m$, i.e., $\reg(I_1^{a_1}\cdots I_m^{a_m})$ is a linear function in  $a=(a_1,\dots, a_m)\in \NN^m$ if   $a_i\gg 0$  for every $i$. We show that this is actually the case when each ideal $I_i$ is generated  by polynomials of a given degree, but false in general. The method of proof allows us to easily generalize the results to ideals in arbitrary standard graded $K$-algebras.

We refer the reader to \cite{BH} for basic commutative algebra.

After the fist version of this note had been uploaded to arXiv.org, M. Chardin informed us that the main result is a special case of a theorem proved by Bagheri, Chardin and H\'a \cite[Throrem 4.6]{BCH}. We hope that our simple approach to the problem and the accompanying examples are nevertheless welcome.

\section{Regularity for powers and products}
For a finitely generated graded non-zero $R$-module $M$ and a nonnegative integer $j\leq \pd_R(M)$ we denote the largest degree of a minimal generator of the $j$-th syzygy module by  $t_j^R(M)$, and, by convention, set $t_j^R(M)=-\infty$ if $j>\pd_R(M)$.  By definition one has $\reg_R(M)=\max\{ t_j^R(M)-j : j\geq 0\}$.

Let $I_1,\dots, I_m$ be non-zero homogeneous ideals of $R$ and let $f_{i1},\dots, f_{ig_i}$ be a minimal homogeneous generating system of $I_i$. Set $d_{ij}=\deg f_{ij}$ and 
$$
B=R[z_{ij} : 1\leq i\leq m,  1\leq j \leq g_i]=K[x_1,\dots, x_n,z_{ij} : 1\leq i\leq m,  1\leq j \leq g_i].
$$ 

To simplify the exposition we set 
$$
I^a=I_1^{a_1}\cdots I_m^{a_m}  \quad\text{for}\quad a\in \NN^m,
$$
and we say that a formula (or a property)  holds for $a\gg 0$ if there exist a $b\in \NN^m$ such that it holds for every $a\in b+\NN^m$.

The multigraded Rees algebra $R(I_1,\dots,I_m)$ can be seen as the subalgebra of $R[t_1,\dots,t_m]$  whose  elements have the form  $\sum_{a\in \NN^m}  F_at^a$ with $F_a\in I^a$. The polynomial ring $R[t_1,\dots,t_m]$ is naturally $\ZZ\times \ZZ^m$-graded if we extend the $\ZZ$-grading on $R$ by setting $\deg(x_i)=(1,0)$ and $\deg t_j=(0,e_j)$ where $e_j$ denotes the $j$-th unit vector in $\ZZ^m$. Evidently $R(I_1,\dots,I_m)$ is a graded subalgebra. It has a structure of  $B$-module via the $R$-algebra map sending $z_{ij}$ to $f_{ij}t_i$. This map is $\ZZ\times \ZZ^m$-graded  if we equip  $B$ with the graded structure  induced by the assignment $\deg(x_i)=(1,0)$ and $\deg(z_{ij})=(d_{ij},e_i)$.

Let $A=K[z_{ij} : 1\leq i\leq m,  1\leq j \leq g_i ]\subset B$ with the induced $\ZZ\times \ZZ^m$-graded structure. Note that $A$ has no elements of degree $(g,0)\in \ZZ\times \ZZ^m$ such that $g\neq 0$. Therefore, given  a finitely generated $\ZZ\times \ZZ^m$-graded $A$-module $M$,  the maximum $\ZZ$-degree in an $a$-graded component, 
$$
\rho_M(a)=\sup\{ i\in \ZZ  : M_{(i,a)}\neq 0\},\quad a\in \ZZ^m,
$$
is finite. The key to the result about the behavior of regularity is a description of $\rho_M(a)$: it is the supremum over the values of finitely many linear functions  $L_k(a)$ where each $L_k$ is defined on a certain subset of $\ZZ^m$. 
In order to describe these subsets we set $[m]=\{1,\dots,m\}$, and
$$
\supp L=\{i\in [m] :\alpha_i\neq 0\}\subset[m]
$$
for a linear polynomial  $L=\sum_{i=1}^m \alpha_iu_i+\alpha_0$. The nonzero coefficients of the relevant linear functions are given by $\ZZ$-degrees $d_{ij}$  of the indeterminates of  $A$:

\begin{lemma} 
\label{lem1}
Let $M$ be a finitely generated $\ZZ\times \ZZ^m$-graded $A$-module. Then there exist $w_1,\dots,w_c\in \ZZ^m$ and linear functions  $L_1,\dots,L_c$ on $\ZZ^m$,
$$
L_k(u)=\sum_{i=1}^m \lambda_{ki} u_i +\lambda_{k0},\qquad \lambda_{k0}\in \ZZ,\quad 
\lambda_{ki}\in \{0\}\cup \{d_{i1}, \dots, d_{ig_i} \},\  i=1,\dots, m
$$ 
such that 
$$
\rho_M(a)=\sup\Bigl\{ L_k(a) :  1\leq k\leq c \mbox{ and }  a\in w_k+\sum_{i\in \supp L_k} \NN e_i\Bigr\}
$$
for every $a\in \ZZ^m$. In particular we have 
$$
\rho_M(a)=\sup\bigl\{ L_k(a) :   1\leq k\leq c \mbox{ and } \supp L_k=[m] \bigr\} \quad\text{for}\quad a\gg 0.
$$
\end{lemma} 

\begin{proof}  We represent $M$ as a quotient $F/U$ of the graded free module $F$ by a graded submodule $U$. Then we replace $U$ with its initial submodule $\ini(U)$ with respect to some term order. Next we filter $F/\ini(U)$ so that the successive quotients are shifted copies of quotients of $A$ by monomial prime ideals of $A$. It follows that the (multigraded) Hilbert series 
$$
\HS_M(x,s)=\sum_{(i,a)\in \ZZ\times \ZZ^m}  \dim M_{(i,a)} x^is^a\in \QQ[|x, s_1,\dots, s_m|]
$$ 
can be written as the sum of the Hilbert series  of residue class rings $A/P_k(-v_k,-w_k)$  where $P_k$ an ideal generated by a subset of the variables of $A$ and $(-v_k,-w_k)\in \ZZ\times \ZZ^m$ and is a shift.   That is, 
$$
\HS_M(x,s)=\sum_{i=1}^c x^{v_k}s^{w_k} \HS_{A/P_k}(x,s)
$$
and therefore 
$$
\rho_M(a)=\sup \bigl\{ \rho_{A/P_k}(a-w_k)+v_k : k=1\dots,c\bigr\}.
$$
Set 
$$
S_k=\bigl\{ i : \mbox{ there exists } j \mbox{ such that } z_{ij}\not\in P_k\bigr\}
$$
and 
$$
\lambda_{ki}=\max\bigl\{ d_{ij} : z_{ij}\not\in P_k\bigr\}
$$
if $i\in S_k$ and $\lambda_{ki}=0$ otherwise. 
Since 
$$\HS_{A/P_k}(x,s)=\frac{1}{\prod_{z_{ij}\not\in P_k}  (1-x^{d_{ij}}s_i)}$$ 
we have 
$$\rho_{A/P_k}(a)=
\begin{cases}
\sum_{j=1}^m \lambda_{ki}a_i & \mbox{ if } a\in \sum_{i\in S_k} \NN e_i,\\
-\infty  &  \mbox{ otherwise.} 
\end{cases}
$$
Summing up, we obtain the desired description  where  $L_k$ is defined as 
\begin{equation*}
L_k(u)=\sum_{i=1}^m \lambda_{ki} (u_i-w_{ki}) +v_k.
\end{equation*}
\end{proof}

Set $\mm=(x_1,\dots, x_n)$ and $\cR=R(I_1,\dots, I_m)$.  Consider the Koszul complex $K(\mm, 
\cR)$ and its homology  $H(\mm, \cR)$. By construction $H_j=H_j(\mm, \cR)$ is a finitely generated $\ZZ\times  \ZZ^m$ graded $B$-module  annihilated by $\mm$ and $\dim_K H_j(\mm, \cR)_{(u,a)}=\beta_{ju}(I^a)$. Hence $H_j$ is a finitely generated $\ZZ\times  \ZZ^m$ graded $A$-module and 
$$
\rho_{H_j}(a)=t_j^R(I^a).
$$
Then  Lemma \ref{lem1} applies so that $t_j^R(I^a)$ is given as  the supremum  of finitely many linear functions each of which is  evaluated on a translated (possibly lower dimensional) non-negative orthant. 
Indeed, one can  explicitly compute such a representation by following the steps described in the proof of Lemma \ref{lem1} and applying them to the Koszul homology module $H_i$. 

For simplicity we state only the main consequence about the asymptotic behavior of  $t_i(I^a)$ and $\reg(I^a)$. It is of course determined by the linear functions $L_k$ with $\supp L_k=[m]$.

\begin{theorem}
\label{thm1}
Let $I_1,\dots, I_m$ be given as above. Then, for every $0\leq j \leq n-1$, there exist $b_j\in \NN$ and  linear functions  $L^{(j)}_k(u)=\sum_{i=1}^m \lambda^{(j)}_{ki} u_i +\lambda^{(j)}_{k0}$  with  $1\leq k\leq b_j$ such that $\lambda^{(j)}_{ki}\in   \{d_{i1}, \dots, d_{ig_i} \}$ and   $\lambda^{(j)}_{k0} \in \ZZ$  and such that 
$$
t_j^R(I^a) =\sup\{ L^{(j)}_k(a) :  1\leq k\leq b_j  \} \quad \mbox{ for } a\gg 0.
$$
In particular, $\pd_R(I^a)$ is constant for $a\gg 0$ and 
$$\reg(I^a) =\sup\{ L^{(j)}_k(a)-j  : 0\leq j\leq n-1 \mbox{ and } \ 1\leq k\leq \ b_j  \} \quad \mbox{ for } a\gg 0. $$
\end{theorem}

If $I_i$ is generated in a single degree, say $d_i$, then there is only one choice for $\lambda^{(j)}_{ki}$, $i>0$, namely $d_i$. If this holds for all ideals $I_i$, then the supremum is obviously given by a single linear polynomial, and we obtain

\begin{corollary}
\label{thm2}
If  each $I_i$ is generated in a single degree, say $d_i$, then $\pd_R(I^a)$ is constant for $a\gg 0$, say equal to $p$, and for each $j$, $0\leq j \leq p$, there exists a linear polynomial  $L^{(j)}(u)=\sum_{i=1}^m d_i u_i +\lambda^{(j)}_{0}$    with $\lambda^{(j)}_{0} \in \ZZ$  such that 
$$
t_j^R(I^a) =L^{(j)}(a) \quad \mbox{ for } a\gg 0.
$$
In particular, there exists an linear polynomial $L(u)=\sum_{i=1}^m d_i u_i +\lambda_{0}$ such that 
$$
\reg(I^a) =L(a) \quad \mbox{ for } a\gg 0.
$$
\end{corollary}

The conclusion of  the corollary holds for a single ideal without any hypothesis on the degrees of the generators since the supremum of finitely many linear functions on $\ZZ$ is always given by a single one for large values of the argument.

\begin{remark} 
Let $J_i$ be a homogeneous reduction of $I_i$ for every $i$.  Then $\cR=R(I_1,\dots, I_m)$ is a finitely generated algebra over $R(J_1,\dots, J_m)$. It follows that the Koszul homology $H(x,\cR)$ is finitely generated over a polynomial ring whose variables have degree $(v,e_i)\in \ZZ\times \ZZ^m$ where $v$ varies in the set of degrees of the generators of $J_i$. Hence the coefficients $\lambda^{(j)}_{ki}$ in Theorem \ref{thm1} can be taken in the set of the degrees of the generators of $J_i$. Therefore Corollary \ref{thm2} holds more generally if each $I_i$ has a reduction generated in  a single degree.   Note that Kodiyalam proved in \cite{Kod} that for a single ideal $I_1$ one has $\reg(I_1^a)=ad+b$ for $a\in \NN$  and $a\gg 0$ where $d$ is the minimum of  $t_0(J)$ where $J$ varies in the set of the homogeneous reduction of $I_1$.
\end{remark}

\begin{remark}
Under the assumption that $I_i$ is generated in degree $d_i$, in \cite[Remark 3.2.]{BC4} we have observed that $\reg(I^a)\leq \sum_{i=1}^m d_i a_i +\reg_0 \cR$ for all $a\in \NN^m$ where  $\reg_0 \cR$ is the regularity computed according to the $\ZZ$-graded structure of   
$\cR=R(I_1,\dots, I_m)$ induced by inclusion  $R\to \cR$.  
So the constant $\lambda_{0}$ in Corollary \ref{thm2} is  $\leq  \reg_0 \cR$. But in general the inequality can be strict. For example, already for  $m=1$ if  $I_1$ is an ideal such that $\{ a\in \NN :  \reg I_1^a>ad_1 \}$ is finite and not empty, then $\lambda_{0}=0$ and $\reg_0 \cR >0$. Examples of ideals of this kind   are described in  \cite{Bor} and  \cite{Con}. 
\end{remark}

\begin{remark}
So far we have considered ideals in polynomial rings. If $R$ is not a polynomial ring (but still standard graded), then the projective dimension   and the regularity of $R$-modules computed via syzygies can be infinite. Nevertheless, one has a bound for $t_j^R(I^a)$ for all $j$ as in Theorem \ref{thm1}. Instead of the Koszul complex, the resolution of $K$ over the polynomial ring, one must use the free resolution of $K$ over $R$ (of infinite length if $R$ is not a polynomial ring).
\end{remark}

In order to generalize Theorem \ref{thm1} and Corollary \ref{thm2} to a general result for ideals in an arbitrary standard graded algebra $R$, we consider the  regularity based on local cohomology. It is defined by
$$
\Reg_R(M)=\sup \{i+j: H_\mm^i(M)_j\neq 0\}
$$
When $R$ is a standard graded  polynomial ring over a field, then $\reg_R(M)=\Reg_R(M)$ by a theorem of Eisenbud and Goto (for example, see \cite[Theorem 4.3.1]{BH}).

\begin{theorem}
The conclusions of Theorem \ref{thm1} and Corollary \ref{thm2} hold for ideals in arbitrary standard graded $K$-algebras, if $\reg$ is replaced by $\Reg$.
\end{theorem}

\begin{proof}
We write $R$ as a residue class ring of a standard graded polynomial ring $S$ over $K$. Then $\Reg_R(I^a)=\Reg_S(I^a)=\reg_S(I^a)$. Therefore it can be computed from the Koszul homologies $H_j(\mm,I^a)$ where $\mm$ is now the graded maximal ideal of $S$. 

We set $\cR=R(I_1,\dots,I_m)$ as above. Then $H_j(\mm,I^a)$ is the degree $a$ component of $H_j(\mm,\cR)$ in the $\ZZ^m$-grading.  After this observation the proof of Theorem \ref{thm1} can be copied.
\end{proof}

\begin{remark}
In \cite{Gh} the author provides a  linear upper bound for $\Reg(I^a)$ where $I_1,\dots I_m$ are ideals in a standard graded algebra $R$ over an Artinian local ring.
\end{remark}

\begin{remark} 
The results above can be immediately generalized to the  behavior of $\reg(I^aM)$ (or $\Reg(I^aM)$) where $M$ is a finitely generated graded $R$-module. Instead of the multi-Rees $\cR$ alone, one considers the finitely generated $\cR$-mdodule
$$
\MM=M(I_1,\dots,I_m)=\bigoplus_a I^aM.
$$ 
The Koszul homology $H_j(\mm, \MM)$ is a finitely generated $A$-module, and the remaining arguments need no change, The case $m=1$ is due to Trung and Wang \cite{TW}.
\end{remark}

\section{Examples} 

We now discuss an example showing that, in general, $t_j^R(I^a)$ and $\reg(I^a)$ need not be linear functions of $a$ for $a\gg 0$.
 
\begin{example} Consider the ideals  $I_1=(x,y^2)$, $I_2=(x^2,y)$ and  in the polynomial ring $R=K[x,y]$.
One has 
$$
t_i(I_1^{a_1}I_2^{a_2})=
\begin{cases}
\max\{2a_1+a_2, a_1+2a_2\}+i  & \mbox{ if } i=0,1 \mbox{ and } (a_1,a_2)\in \NN^2\setminus\{(0,0)\},\\
0  &\mbox{ if }  i=0 \mbox{ and } (a_1,a_2)= (0,0).
\end{cases}
$$ 
So that 
$$
\reg(I_1^{a_1}I_2^{a_2})=\max\{2a_1+a_2, a_1+2a_2\} \mbox{ for all } (a_1,a_2)\in \NN^2.
$$

To establish these formulas we propose two approaches. First we  follow the strategy outlined above and deduce the result from the study of the multigraded Hilbert series of the Koszul homology modules $H_i(x,y,R(I_1,I_2))$.   Secondly we  write down  the resolution of $I_1^{a_1}I_2^{a_2}$ directly.    

The presentation of the Rees algebra  $\cR=R(I_1,I_2)$ is given by:  
$$
B/P\simeq \cR,\qquad
B=R[z_{11},z_{12},z_{21}, z_{22}],
$$ 
with $\ZZ\times \ZZ^2$ graded structure induced by $\deg(x)=\deg(y)=(1,0,0)$, 
$\deg(z_{11})=(1,1,0)$,  $\deg(z_{12})=(2,1,0)$, 
$\deg(z_{21})=(2,0,1)$,  $\deg(z_{22})=(1,0,1)$. 
Here 
$$
P=(y^2z_{11}-xz_{12},  yz_{21}-x^2z_{22}, xyz_{11}z_{22}-z_{12}z_{21}).
$$

Set $A=K[z_{11},z_{12},z_{21}, z_{22}]$ so that  $H_0=H_0(x,y, \cR)=\cR/(x,y)=A/(z_{12}z_{21})$. Using a filtration or by direct computation one has:

\begin{equation}
\label{partdec1}
\HS_{H_0}(x,s_1,s_2)=\frac{1}{(1-xs_1)(1-x^2s_1)(1-xs_2)}+\frac{x^2s_2}{(1-xs_1)(1-xs_2)(1-x^2s_2)}.
\end{equation} 

From this we deduce that 
$$
t_0(I_1^{a_1}I_2^{a_2})=\max\bigl\{  2a_1+a_2 : a\in \NN^2 ,  a_1+2a_2 : a\in e_2+\NN^2  \bigr\},
$$ 
that is, 
$$
t_0(I_1^{a_1}I_2^{a_2})=\max\{  2a_1+a_2 , a_1+2a_2  \}\quad \mbox{for all } a\in \NN^2.
$$

Similarly for $H_1=H_1(x,y, \cR)$ one obtains: 

\begin{equation}
\label{partdec2}
\HS_{H_1}(x,s)=\frac{x^3s_1 }{ (1-xs_1)(1-x^2s_1)(1-xs_2)}+\frac{x^3s_2 }{ (1-xs_1)(1-x^2s_2)(1-xs_2)}. 
\end{equation}

Then it follows that 
$$
t_1(I_1^{a_1}I_2^{a_2})=\sup\bigl\{  2a_1+a_2+1 : (a_1,a_2)\in e_1+\NN^2,   a_1+2a_2+1 : (a_1,a_2)\in e_2+\NN^2\bigr\}.
$$
Hence 
$$
t_1(I_1^{a_1}I_2^{a_2})=\max\bigl\{  2a_1+a_2+1, a_1+2a_2+1\bigr\}
$$ 
for all $a\in \NN^2\setminus \{(0,0)\}$.

Let us now sketch the approach via free resolutions. First one checks that $I_1^{a_1}I_2^{a_2}$ is minimally generated by the $1+a_1+a_2$ elements
$$
\begin{array}{cl}  
    x^{a_1}y^{a_2}, &  x^{a_1-1}y^{a_2+2}, x^{a_1-2}y^{a_2+4}, \dots, y^{a_2+2a_1}, \\
  \verteq &  \\
    x^{a_1}y^{a_2}, &  x^{a_1+2}y^{a_2-1}, x^{a_1+4}y^{a_2-2}, \dots, x^{a_1+2a_2}.
\end{array} 
$$
This confirms our formula
$$
t_0(I_1^{a_1}I_2^{a_2})=\max\{  2a_1+a_2 , a_1+2a_2  \}\quad \mbox{for all } a\in \NN^2.
$$
By Hilbert's syzygy theorem $I_1^{a_1}I_2^{a_2}$ has a free resolution of length $1$ (if $a_1+a_2>0$). We claim the syzygy module of $I_1^{a_1}I_2^{a_2}$ is generated by the rows of the following $(a_1+a_2)\times(a_1+a_2+1)$ matrix
$$
\phi=
\begin{pmatrix}
y^2&-x\\
&y^2&-x\\
&&\ddots&\ddots\\
&&& y^2&-x\\
x^2&&&&&-y\\
&&&&&x^2&-y\\
&&&&&&x^2&-y\\
&&&&&&&\ddots&\ddots\\
&&&&&&&&x^2&-y
\end{pmatrix}
$$
In fact, the rows of $\phi$ are syzygies. Therefore we have a complex
$$
0\to R^{a_1+a_2}\xrightarrow{\phi}\to R^{1+a_1+a_2}\to I_1^{a_1}I_2^{a_2} \to 0.
$$
Both $x^{a_1+2a_2}$ and $y^{2a_1+a_2}$ appear as maximal minors of $\phi$. Thus the ideal of maximal minors of $\phi$ has grade $2$ in $R$. Now the Buchsbaum-Eisenbud exactness criterion (for example, see \cite[1.4.3]{BH}) implies that our complex is exact, and so we know the first syzygy module of $I_1^{a_1}I_2^{a_2}$. 

Keeping track of the degrees one obtains that 
$$
t_1(I_1^{a_1}I_2^{a_2})=\max\bigl\{  2a_1+a_2+1, a_1+2a_2+1\bigr\}
$$ 
for all $a\in \NN^2\setminus \{(0,0)\}$.

\end{example} 

\begin{example}
In the example above the regularity and  the invariants $t_i(I^a)$ are  given by the supremum of two linear functions for $a\gg 0$. But even for $m=2$ the number of linear functions can be larger. An example is given by the ideals $I_1=(x, y^2, z^3)$ and $I_2=(x^4, y^3, z)$ in $R=K[x,y,z]$ for which each $t_i(I_1^{a_1}I_2^{a_2})$ and the regularity is given asymptotically by the maximum of  $3$ linear functions.  More precisely for every $a\geq (1,1)$ one has: 
$$\begin{array}{rlllll}
t_0(I_1^{a_1}I_2^{a_2})&=&\max\{ a_1+4a_2+1 ,& 2a_1+3a_2,     & 3a_1+a_2    &\},\\ \\
t_1(I_1^{a_1}I_2^{a_2})&=&\max\{ a_1+4a_2+2, & 2a_1+3a_2+1, &3a_1+a_2+2 &\},\\ \\
t_2(I_1^{a_1}I_2^{a_2})&=&\max\{ a_1+4a_2+3 ,&2a_1+3a_2+2,  &3a_1+a_2+3  & \},\\ \\
\reg(I_1^{a_1}I_2^{a_2})&=&\max\{a_1+4a_2+1 ,&2a_1+3a_2,      &3a_1+a_2+1  &\}.
\end{array}
$$
The computation has been done following the strategy outlined above, making use of Macaulay 2 \cite{M2} and CoCoA \cite{CoCoA} for the computation of the  Hilbert series of the Koszul homology $H_i(x,y,z,R(I_1,I_2))$ and of the filtration allowing to decompose the Hilbert series as a sum of rational series with positive numerators and ``partial" denominators similarly to what we have done in   (\ref{partdec1}) and (\ref{partdec2}).  
\end{example}
We expect that there is no bound on the number of linear functions, at least if one allows an arbitrary number of indeterminates.
\bigskip

\emph{Acknowledgment.}\enspace The authors are grateful to Alessio Sammartano for his help in using Macaulay 2 and to Marc Chardin for pointing out the results in \cite{BCH}.

\end{document}